\documentclass[a4paper,12pt,titlepage, doublespacing]{report}
\usepackage[T1]{fontenc}
\usepackage{textcomp} 
\usepackage{amsmath, amsthm, amsfonts, amssymb}
\usepackage{marvosym}
\newcommand{\E}{{\mathbb{E}}}
\newcommand{\R}{{\mathbb{R}}}
\newcommand{\D}{{\mathbb{D}}}

\newcommand{\N}{{\mathbb N}}
\newcommand{\C}{{\mathbb C}}

\newcommand{\ve}{\varepsilon}

\newcommand{\vp}{\varphi}

\newcommand{\lmd}{{\lambda }}
\newcommand{\const}{\mathop{\rm const}}

\newcommand{\Supp}{\mathop{\rm Supp}}

\theoremstyle{plain}
\newtheorem{thm}{Theorem}
\newtheorem{lem}{Lemma}

\newtheorem{remark}{Remark}
\newtheorem{corollary}{Corollary}
\newtheorem{defn}{Definition}
\newtheorem{prop}{Proposition}

\begin{document}
\begin{center}
\textbf{Refinements of asymptotics at zero of Brownian self-intersection local times.}
\end{center}

\begin{center}
   \textbf{A. A. Dorogovtsev} \\
   Institute of Mathematics, National Academy of Sciences of Ukraine, Ukraine \\
E-mail address : \textit{andrey.dorogovtsev@gmail.com }

\end{center}

\begin{center}
\textbf{Naoufel Salhi} \\
IPEIN, University of Carthage, Tunisia \\
E-mail address : \textit{ salhi.naoufel@gmail.com }

\end{center}

\begin{center}
\textbf{Abstract }
\end{center}

In the article we establish some estimates related to the Gaussian densities and to Hermite polynomials in order to obtain an almost sure estimate for each term of the It\^{o}-Wiener expansion of the self-intersection local times of the Brownian motion. In dimension $d\geqslant 4$ the self-intersection local times of the Brownian motion can be considered as a family of measures on the classical Wiener space. We provide some asymptotics relative to these measures. Finally, we try to estimate the quadratic Wasserstein distance between these measures and the Wiener measure. 
\begin{flushleft}
\textit{2010 Mathematics Subject Classification } 60G15, 60H05, 60H40, 60J55 \\

\textit{Key words :} self-intersection local times, Brownian motion, It\^{o}-Wiener expansion, classical Wiener space, Sobolev spaces, generalized Wiener functions, capacity, Wasserstein distance, Talagrand inequality \\
\end{flushleft}

\section*{Introduction}

$\quad$Local times of self-intersections of Brownian motion is a rich subject that was investigated by many mathematicians. This paper is too small to cover all the previously obtained results. But we can remind the reader about some significant results -according to the point of view of the authors- that may be classified in the same topic with the results of this paper. A. Dvoretzky, P. Erdos and S. Kakutani \cite{DEK1,DEK2,DEK3,DEK4} were the first to discuss the existence of self-intersections of the Brownian motion. For more known facts in this field, we can refer the reader to \cite{AOmono} and the survey paper \cite{AOsurv} by A.A. Dorogovtsev and O.L. Izyumtseva.  \\

\noindent Among the reasons that explain the importance of this theory we find the relationship with the investigation of polymers. Polymers are long chains of small molecules (monomers) that are chemically connected. These huge molecules can be found in nature (proteins, DNA, RNA ...) and can be manufactured for industrial purposes. Brownian motion was proposed by S.F. Edwards \cite{Edwards} as a good mathematical model for self-avoiding polymers. But due to the "excluded volume effect" the study of local times of self-intersections was a necessity. In addition, local times of self-intersections are crucial tools for the study of the geometry of trajectories of the Brownian motion \cite{LeGall}. \\

\noindent Let $w(t)=\big( w_1(t),\cdots,w_d(t)\big),\; 0\leqslant t\leqslant 1,$ be a $d$-dimensional Brownian motion. The local times of self-intersections are formally denoted by : 
\begin{equation} \label{1}
\int _{\Delta_2} \delta_0 \big(w(t)-w(s) \big)  dsdt
\end{equation}
where $\Delta_2= \{ (s,t)\in [0,1]^2\, ;\, s<t\}. $ Rigorously, this object can be defined as the $L^p$- limit :
\begin{equation} \label{2} 
\mathrm{L^p}-\lim _{\ve \to 0}  \int _{\Delta_2} p^d_{\ve} \big(w(t)-w(s) \big)  dsdt ,
\end{equation}
for some $p\geqslant 1$. Here 
\begin{equation} \label{21} 
 p^d_{\ve}(z)=\frac{1}{ (2\pi \ve)^{d/2} } \exp \left( \frac{-\|z\|^2}{2\ve} \right),\; z\in \R^d,\, \ve >0
\end{equation}
\noindent It was proved that local times of self-intersections exist only for one-dimensional Brownian motion. However, the previous definitions can be extended as follows :
\begin{equation} \label{3}
\begin{split}
\rho_2(u)
&= \int _{\Delta_2} \delta_u \big(w(t)-w(s) \big)  dsdt\\
&=\mathrm{L^p}-\lim _{\ve \to 0}  \int _{\Delta_2} p^d_{\ve} \big(w(t)-w(s)-u \big)  dsdt,\, u\in \R^d .
\end{split}
\end{equation}
$\rho_2$ can be considered as the density of the occupation measure : 
\begin{equation} \label{4}
\mu _2 (A )=\int _{\Delta_2} \mathbf{1}_{A } \big(w(t)-w(s)\big)  dsdt ,\; A \in \mathcal{B}\big( \R^{d} \big)
\end{equation}
\noindent In such case, the following relation, called by "occupation formula", 
\begin{equation} \label{5}
\int _{\R^{d}} f(u) \rho _2(u)du =\int _{\Delta_2} f \big(w(t)-w(s) \big)  dsdt \quad  a.s.
\end{equation}
holds for every $f :\R^{d} \to \R$ positive (or bounded) and measurable. In \cite{Rosen}, J. Rosen proved that if $d=2$ or $d=3$ then the measure $\mu_2$ has a square integrable density which is continuous in $\R^d\setminus \{0\}$ and that $\mu_2$ does not have a density for $d\geqslant 4$. The basic tool there was the Fourier transform. \\ 

\noindent The behavior of the local time $\rho_2$ at the origin was investigated by some mathematicians. In the case of the planar Brownian motion, Varadhan \cite{Varadhan} and Le Gall \cite{LeGall} proved that the following limit 
\begin{equation} \label{6}
 \lim _{u\to 0} \, \rho_2(u)-\dfrac{1}{\pi}\log \left( \dfrac{1}{ \|u\| } \right)   
\end{equation}
exists almost surely and in $L^2$. In \cite{Streit2}, L. Streit et al. estimated the rate of convergence of the random variables 
$$ L_{\ve, c}=\int _{\Delta_2} p^2_{\ve} \big(w(t)-w(s)\big)  dsdt -\E \int _{\Delta_2} p^2_{\ve} \big(w(t)-w(s)\big)  dsdt $$
when $\ve \to 0$. Precisely, they proved that for any $\alpha <1$ there exists a constant $C_{\alpha} >0$ such that for all $\ve >0$ 
$$ \E \Big( L_{\ve, c}-L_c \Big) ^2 \leqslant C_{\alpha} \, \ve ^{\alpha} $$
where $\displaystyle L_c=\lim _{\ve \to 0} L_{\ve, c}$. In \cite{Streit}, an estimation of the speed of convergence of the gap-renormalized double self intersection local times of the planar Brownian motion was obtained. \\

\noindent In the $3$-dimensional case, M. Yor proved in \cite{Yor} that the following limit 
\begin{equation} \label{7}
\lim _{u\to 0}  \dfrac{1}{ \sqrt{ \log   \dfrac{1}{ \|u\| }     } } \Big( \rho_2(u)-\dfrac{1}{2\pi \|u\| }  \Big)
\end{equation}
exists in $L^2$. This clearly shows that in dimensions $2$ and $3$, the local time $\rho_2$ blows up at the origin in different scales. The expressions \eqref{6} and \eqref{7} are called renormalized self-intersections local times. A different form of renormalization related to the planar Brownian motion was proposed by E. B. Dynkin in \cite{Dynkin}. Denote by $\Delta_k$ the set $ \{ (t_1,\cdots,t_k)\in [0,1]^k\, ;\, t_1<t_2\cdots < t_k\}. $ Put 
\begin{equation}
T_{k,\ve}^{\vp}=\int_{\Delta_k} \prod_{i=1}^{k-1} q_{\ve}\big(w(t_{i+1})-w(t_i)\big)\, \vp(\vec{t})\, d\vec{t} 
\end{equation} 
where $\vp \in C(\Delta_k)$, which means it is a continuous real-valued function defined on $\Delta_k$, and $q_{\ve},\, \ve >0$ is defined as follows 
$$\begin{cases}
 q_{\ve}(x)=\dfrac{1}{\ve ^2} \, q\big( \frac{x}{\ve} \big)\\
 q \, \text{ is a probability density on } \R^2 \\
 \int_{ \R^2 } \big ( \log |x| \big ) ^{j}q(x)dx <\infty \quad \forall \quad j\in \N ^* \\
 \exists \; \beta >0 \text{ such that }  \int_{ \R^2 } e^{\beta |x|}q(x)dx <\infty  \,.
\end{cases}$$
\noindent Now define an operator : $ B_k^l : C(\Delta_k)\to C(\Delta_l)$ by :
$$ B_k^l(\vp)(s_1,\cdots,s_l)=\sum _{\sigma } \vp( s_{\sigma (1)},\cdots,s_{\sigma (k)})  $$
where the sum is taken over all surjective maps $\sigma : \{1,\cdots,k\} \to \{1,\cdots,l\}$ such that $\sigma(i)\leqslant \sigma(j)$ whenever $i<j$. Dynkin proved that for any $k\geqslant 2,\, \vp \in C(\Delta_k)$ and any $p\geqslant 1$ the following limit 
\begin{equation}
\mathrm{L^p}-\lim _{\ve \to 0} \sum_{l=1}^k \Big( \dfrac{\log \ve}{2\pi} \Big)^{k-l} T_{k,\ve}^{B_k^l(\vp)}  
\end{equation} 
exists.\\
\noindent The asymptotics \eqref{6} and \eqref{7} were obtained later by P. Imkeller, V. Perez-Abreu and J. Vives \cite{IPA} through a different approach. They used the white noise theory. Precisely, the core of their work was the It\^{o}-Wiener expansion (or chaotic expansion). It was established that for $d=2$ or $d=3$, the density $\rho_2(u), \, u\neq0$ has the following It\^{o}-Wiener expansion :
\begin{equation} \label{8}
\begin{split}
\rho_2(u)
&= \sum_{k=0}^{\infty} \sum_{n_1+..+n_d= k} \int_{\Delta_2 } \; \prod _{j=1}^d \frac{1}{\sqrt{ n_j !} } \, H_{n_j} \Big ( \frac{ w_j(t)-w_j(s)  }{\sqrt{t-s} } \Big ) \times \\
&\phantom{ = \sum_{k=0}^{\infty} \sum_{n_1+..+n_d= k} \int_{\Delta_2 } \; \prod _{j=1}^d \frac{1}{\sqrt{ n_j !} } \,  } \times \, H_{n_j}  \Big( \frac{ u_j  }{\sqrt{t-s} } \Big) \;  p^d_{t-s}(u)ds dt \,.
\end{split} 
\end{equation}
\noindent Here $H_n$ denotes the Hermite polynomial : 
\begin{equation} \label{80} 
H_{n}(x)=(-1)^{n} e^{\frac{x^{2}}{2}} \left( \frac{d}{dx}  \right) ^{(n)} e^{ \frac{-x^{2}}{2}},\, x\in \R
\end{equation}

\noindent The It\^{o}-Wiener expansion of a square integrable random variable which is measurable with respect to the Brownian motion (we denote the space of such kind of random variables by $L^2\big( \Omega, \sigma(w), \mathbb{P}\big)$) has the form  $\eta  = \sum _{k=0}^{\infty} I_k(f_k)$. The summands $I_k(f_k)$ are multiple It\^{o} stochastic integrals of the deterministic square integrable kernels $f_k$ with respect to $w_1,\cdots,w_d$. This series consists of orthogonal summands and converges in $L^2$ sense.\\

\noindent In the same paper, the case of dimension $d\geqslant 4$ was also considered. It was proved that the non-existent local time $\rho_2(u)$ can be replaced by a generalized Wiener function (or a distribution), i.e. an element of some Sobolev space with a negative index. More precisely, the expansion \eqref{8} does not converge in $L^2$ unless we divide the $k$-th term by a fixed positive power of $(k+1)$. 
\begin{thm}\cite{IPA}
Let $d\geqslant 4$ and $u\in \R^d,\,u\neq 0$. The limit 
\begin{equation} \label{81}  
\rho_2 (u)=\lim_{\varepsilon \to 0} \int_{\Delta_2  }  p^d_{\varepsilon}\big( w(t)-w(s)-u\big)dsdt  
\end{equation}
exists in every Sobolev space $\D^{2,\gamma}$ for $\gamma< (4-d)/2$, and is given by the distribution valued series :
\begin{equation} \label{9}
\begin{split}
\rho_2(u)
&= \sum_{k=0}^{\infty} \sum_{n_1+..+n_d= k} \int_{\Delta_2 } \; \prod _{j=1}^d \frac{1}{\sqrt{ n_j !} } \, H_{n_j} \Big ( \frac{ w_j(t)-w_j(s)  }{\sqrt{t-s} } \Big ) \times \\
&\phantom{ = \sum_{k=0}^{\infty} \sum_{n_1+..+n_d= k} \int_{\Delta_2 } \; \prod _{j=1}^d \frac{1}{\sqrt{ n_j !} } \,  } \times \, H_{n_j}  \Big( \frac{ u_j  }{\sqrt{t-s} } \Big) \;  p^d_{t-s}(u)ds dt \,.
\end{split} 
\end{equation}
\end{thm}
\noindent The space $\D^{2,\gamma}$ is the completion of the following space :
$$ \left\lbrace   \eta=\sum_{k=0}^n I_k (f_k) \in L^2\big( \Omega, \sigma(w),\mathbb{P}\big),\, n\in \N\right\rbrace  $$
with respect to the norm : 
$$ \big \| \eta\big \| _{2,\gamma} ^2=\sum_{k=0}^{n} (k+1)^{\gamma } \, \E \, I_k(f_k)^2 \,.  $$

\noindent When $\alpha <0$, the elements of $\D^{2,\alpha}$ are called generalized Wiener functionals.\\
\noindent If $0<\alpha <\beta $ then
$$  \D^{2,\beta} \subset \D^{2,\alpha} \subset \D^{2,0}=L^2\big( \Omega, \sigma(w),\mathbb{P}\big) \subset \D^{2,-\alpha} \subset \D^{2,-\beta} . $$
\noindent $\D^{2,\pm\infty}$ are defined by 
$$ \D^{2,+\infty}=\cap_{\alpha >0} \D^{2,\alpha}\,,\; \D^{2,-\infty}=\cup _{\alpha >0} \D^{2,-\alpha}\,. $$
\noindent If $n\in \N^*$ then $\D^{2,n}$ coincides with the space of square random variables $\eta \in L^2\big( \Omega, \sigma(w),\mathbb{P}\big)$ which are $n$-times stochastically differentiable and such that $\eta, D\eta, \cdots, D^n\eta $ are square integrable \cite{Nualart, AAD}. Moreover, the norm $\big \| . \big \| _{2,n} $ is equivalent to the norm $\sqrt{\sum_{k=0}^n \big \| D^k\eta \| _{2}^2 }\; $ \cite{Nualart, Sugita}.\\

\noindent Our investigation is in some sens a continuation of \cite{IPA}. We are interested in the sample path behavior of $\rho_2(u)$ at the origin. We established an almost sure upper bound for each term of the chaotic expansion of $\rho_2(u)$. For dimensions $d\geqslant 4$ we relied on a result from \cite{Sugita} in order to deal with $\rho_2(u)$ as a measure on the space of trajectories (the Wiener space : $ W_0^d=\left\lbrace \omega :[0,1]\to \R^d,\, \text{continuous},\, \omega(0)=0\right\rbrace  $). The asymptotics in this case are expressed through potential theory. We succeeded to obtain some asymptotics for the capacity of the support of this measure.\\

\noindent The paper is organised as follows. In section 1 we start by establishing some estimates related to the Gaussian density $p_{t-s}^d$ and to Hermite polynomials in order to obtain an almost sure upper bound for each term of the It\^{o}-Wiener expansion of $\rho_2(u)$ in any dimension. In section 2 we first establish the fact that, when $d\geqslant 4$, the distribution $\rho_2(u)$ can be considered as a measure $\theta_u$ in the Wiener space. We then precise a specific closed set that contains the support of $\theta_u$. After proving some useful properties related to conditional expectation and second quantization operator we establish a formula that indicates how to integrate with respect to $\theta_u$. Finally, we find a lower bound for the capacity of the support of $\theta_u$. In section 3 we try to estimate the quadratic Wasserstein distance between the measure $\theta_u$ and the Wiener measure. We mainly apply the Talagrand inequality to the finite dimensional approximations of these measures.\\  

\section{Estimation of multiple Wiener integrals}
\noindent The main estimate established in this section is based on the following three technical estimates.\\ 
\begin{prop} \label{det.est.} Let $\alpha \in \R_+$ and $u\in \R^d,\, d\geqslant 2$. Then, when $u\to 0$, 
$$\int_{\Delta_2   } \dfrac{1}{(t-s)^{\alpha  }}   \; p^d_{t-s}( u) dsdt
\sim   
\begin{cases}
\dfrac{2^{\alpha  -1}\Gamma (\alpha  +d/2 -1)  }{ \pi ^{d/2} \|u\|^{2\alpha  +d-2}   } \; & \text{if } \alpha  >1- \dfrac{d}{2}\\
\dfrac{2^{\alpha  }}{ \pi ^{d/2}    }\log \Big( \dfrac{1}{ \|u\| } \Big) \; & \text{if } \alpha  = 1-\dfrac{d}{2}\,,

\end{cases}
$$
where $\Gamma$ denotes the Gamma function.
\end{prop}
\begin{proof}
By a change of variables : $(x,y)=(t-s,s) \in [0,1]^2,\, x+y\leqslant 1$ : 
\begin{align*}
\int_{\Delta _2 } \dfrac{1}{(t-s)^{\alpha  }}   \; p^d_{t-s}( u) dsdt 
&= \dfrac{1}{(2\pi) ^{d/2} }\int_{\substack{(x,y)\in [0,1]^2\\ x+y \leqslant 1 }} \dfrac{1}{x^{\alpha  +d/2 } } \; \exp \left( \frac{-\|u\|^2}{2x} \right)dxdy \\ 
&=\dfrac{1}{(2\pi) ^{d/2} } \int_0^1 \dfrac{1-x}{x^{\alpha  +d/2} } \; \exp \left( \frac{-\|u\|^2}{2x} \right)dx \,, \;  z=\frac{\|u\|^2}{2x}\,, \\
&=\dfrac{ 2^{\alpha  -1} }{\pi ^{d/2} \|u\|^{2\alpha  +d-2} } \int _{\frac{\|u\|^2}{2}} ^{\infty}  z^{\alpha  +d/2-2}\Big( 1-\frac{\|u\|^2}{2z}\Big)  \, e^{-z} dz  \,.
\end{align*}
Using the following estimates, when $u\to 0$, 
$$\int _{\frac{\|u\|^2}{2}} ^{\infty}  z^{r } e^{-z} dz 
\sim   
\begin{cases}
 \Gamma (r +1)  \; & \text{if } r  >-1\\
2\log \Big( \dfrac{1}{ \|u\| } \Big) \; & \text{if } r = -1\\
-\dfrac{\|u\|^{2r+2} }{ 2^{r +1}(r +1) }\; & \text{if } r < -1\,,
\end{cases}
$$
we deduce the result.
\end{proof}
$\,$

\begin{prop} \label{wiener.est}
Let $W$ be the one-dimensional Brownian motion and $0\leqslant s <t \leqslant 1$. Then, almost surely,
\begin{equation} \label{115}
\Bigg | H_{n} \Big ( \frac{ W(t)-W(s)  }{\sqrt{t-s} } \Big )\Bigg | \leqslant \dfrac{n !\sqrt{e}  }{(t-s)^{\frac{n}{2}} }   \, \exp \Big( \Big|  W(t)-W(s)\Big| \Big) 
\end{equation}

\end{prop}

\begin{proof}

Using the generating function of Hermite polynomials we obtain the following power series expansion of the holomorphic function $\exp \Big( z \big(  W(t)-W(s)\big) -\frac{z ^2}{2} (t-s) \Big)$ :   
$$\displaystyle \exp \Big( z \big(  W(t)-W(s)\big) -\frac{z ^2}{2} (t-s) \Big)=\sum_{n=0}^{\infty} \dfrac{ (t-s)^{\frac{n}{2}} H_{n} \Big (  \frac{ W(t)-W(s)  }{\sqrt{t-s} } \Big )  }{n!} z^n,\; z\in \C.$$ 
Using the Cauchy integral formula we deduce : 
\begin{align*}
(t-s)^{\frac{n}{2}} H_{n} \Big (  \frac{ W(t)-W(s)  }{\sqrt{t-s} } \Big ) 
&=\dfrac{\partial^n }{\partial z^n}\exp \Big( z \big(  W(t)-W(s)\big) -\frac{z ^2}{2} (t-s) \Big)\Big| _{z=0} \\
&=\dfrac{n!}{2i\pi} \int _{\mathcal{C} } \dfrac{ \exp \Big( z \big(  W(t)-W(s)\big) -\frac{z ^2}{2} (t-s) \Big) }{z ^{n+1}} \,dz      \,.
\end{align*} 
where $ \mathcal{C} =\left\lbrace z\in \C \,;\, |z|=1 \right\rbrace\,.  $\\
It follows that : 
$$ \Bigg | (t-s)^{\frac{n}{2}} H_{n} \Big ( \frac{ W(t)-W(s)  }{\sqrt{t-s} } \Big )\Bigg | \leqslant \dfrac{n!}{2\pi} \int _{\mathcal{C} }   \Bigg | \exp \Big( z \big(  W(t)-W(s)\big) -\frac{z ^2}{2} (t-s) \Big) \Bigg | dz\,. $$
\noindent Letting $z =u+iv,\, u,v\in [-1,1],\, u^2+v^2=1$, we obtain  
\begin{align*}
  &\max _{z \in \mathcal{C} }\Big | \exp \Big( z \big(  W(t)-W(s)\big) -\frac{z ^2}{2} (t-s) \Big)\Big |\\
  &=\max _{u\in [-1,1] }\exp \Big( -u^2(t-s)+ u \big(  W(t)-W(s)\big) +\dfrac{t-s}{2} \Big) \\
  &= \begin{cases} 
        \exp \Big( \dfrac{\big(  W(t)-W(s)\big)^2}{4 (t-s) } +\frac{t-s}{2} \Big) &\text{ if }  \Big | W(t)-W(s)        \Big |  \leqslant 2(t-s)\\
        \exp \Big( \Big | W(t)-W(s) \Big | -\frac{t-s}{2} \Big) &\text{ if }  \Big | W(t)-W(s) \Big | >2(t-s)
     \end{cases} \\
  &\leqslant \sqrt{e} \, \exp \Big( \Big|  W(t)-W(s)\Big| \Big) \,.      
\end{align*}
\noindent The result follows immediately.
\end{proof}
$\;$
\begin{prop}\cite{Sz} \label{Szego}
Let $\alpha \in [1/4,\,1/2]$. Then there exists a constant $c$ such that for any $n\in \N$ and any $x\in \R$  : 
$$
\Big | H_{n} (x)e^{-\alpha x^2} \Big | \leqslant c \, \sqrt{n!} \,n^{-(8\alpha-1)/12} \,.
$$
\end{prop}
$\;$\\
\begin{remark}
We can obtain an almost sure upper bound for $\Big | H_{n} \Big ( \frac{ W(t)-W(s)  }{\sqrt{t-s} } \Big )\Big |$ using proposition \ref{Szego}. If $\alpha \in [1/4,\,1/2]$, then there exists a constant $c$ such that for any $n\in \N$ and any $x\in \R$, almost surely, 
\begin{equation} \label{120}
\Bigg | H_{n} \Big ( \frac{ W(t)-W(s)  }{\sqrt{t-s} } \Big )\Bigg | \leqslant c \, \sqrt{n!} \,n^{-(8\alpha-1)/12} e^{-\alpha \frac{ (W(t)-W(s))^2  }{t-s }} 
\end{equation}
\noindent The advantage in \eqref{115} is that the Brownian increment $W(t)-W(s)$ is separated from the time increment $t-s$. In \eqref{120} we can use the Lévy modulus of continuity to control the ratio $\frac{ W(t)-W(s)  }{\sqrt{t-s} }$, but the estimate will only be valid for small time increments.
\end{remark}
$\;$\\
\noindent Now we can deduce the following almost sure upper bound of each term of the It\^{o}-Wiener expansion given by \eqref{8}.\\

\begin{thm}  Let $Z_i= \displaystyle \max _{ 0\leqslant s \leqslant 1} \big| w_i(s)\big| \,,i=1,..,d,\, d\geqslant 2$. Then, almost surely, when $u\in \R^d,\,u\to 0$,  
\begin{align*} 
&\Bigg |  \int_{\Delta_2 }\prod _{j=1}^d \, H_{n_j} \Big ( \frac{ w_j(t)-w_j(s)  }{\sqrt{t-s} } \Big ) H_{n_j}  \Big( \frac{ u_j  }{\sqrt{t-s} } \Big) \;  p^d_{t-s}(u)ds dt  \Bigg |  \\
&\leqslant \begin{cases} 
 C_0\log \Big( \frac{1}{\|u\|} \Big) \; &\text{if } d=2,\,n_1=n_2 =0\\
 C(n_1,..,n_d) \dfrac{ \exp \Big( 2Z_1+..+2Z_d \Big) }{\|u\| ^{n_1+\cdots+n_d+d-2 } } \; &\text{otherwise. }
 \end{cases} 
\end{align*}
The constants are deterministic.
\end{thm}
\begin{proof} From proposition \ref{Szego} we have :
$$ \left | H_{n}(x)e^{-\frac{ x^2}{4} } \right |\; \leq \; C(n) \quad \forall \; x \in  \R, \; \forall \; n \in \mathbb{N}\,.$$
Combining this with propositions \ref{det.est.} and \ref{wiener.est} we obtain :
\begin{align*} 
&\Bigg |  \int_{\Delta_2  } \left\lbrace  \prod _{i=1}^d H_{n_i} \Big ( \frac{ w_i(t)-w_i(s)  }{\sqrt{t-s} } \Big ) \, H_{n_i}  \Bigg( \frac{ u_i  }{\sqrt{t-s} } \Bigg) \right\rbrace  p^d_{t-s}(u)dsdt \Bigg | \\ 
&\leqslant  \int_{\Delta_2  } \left\lbrace  \prod _{i=1}^d \Bigg |   H_{n_i} \Big ( \frac{ w_i(t)-w_i(s)  }{\sqrt{t-s} } \Big )  \, H_{n_i}  \Bigg( \frac{ u_i  }{\sqrt{t-s} } \Bigg) \exp \Bigg ( -\frac{ u_i ^2 }{4(t-s) } \Bigg)  \Bigg | \right\rbrace  p^d_{t-s}\Big( \frac{u}{\sqrt{2}} \Big)   dsdt  \\
&\leqslant C_0(n_1,\cdots,n_d) \int_{\Delta_2  } \Bigg |   \prod _{i=1}^d H_{n_i} \Big ( \frac{ w_i(t)-w_i(s)  }{\sqrt{t-s} } \Big ) \Bigg |  \, p^d_{t-s}\Big( \frac{u}{\sqrt{2}} \Big) dsdt \\
& \leqslant C_1(n_1,\cdots,n_d)\int_{\Delta_2  } \dfrac{\exp \Big( \sum_{i=1}^d \big|  w_i(t)-w_i(s)\big| \Big)   }{(t-s)^{\frac{n_1+\cdots+n_d}{2}} }   \; p^d_{t-s}\Big( \frac{u}{\sqrt{2}} \Big) dsdt \\
& \leqslant C_1(n_1,\cdots,n_d)\exp \Big( \sum_{i=1}^d 2Z_i \Big) \int_{\Delta_2   } \dfrac{1}{(t-s)^{\frac{n_1+\cdots+n_d}{2}} }   \; p^2_{t-s}\Big( \frac{u}{\sqrt{2}} \Big) dsdt \\
& \leqslant C_2(n_1,\cdots,n_d)
 \begin{cases} 
 \log \Big( \dfrac{1}{\|u\|} \Big) \; &\text{if } d=2,\, n_1+n_2 =0\\
 \dfrac{ \exp \Big( \sum_{i=1}^d 2Z_i \Big) }{\|u\| ^{n_1+\cdots+n_d+d-2 } } \; &\text{otherwise. }
 \end{cases}
\end{align*} 
\end{proof}

\section{Asymptotics for generalized random functionals as a measure on the Wiener space }

\noindent In this section we focus on the case of dimension $d\geqslant 4$. Our first result which consists in considering the local time as a measure on the Wiener space is based on theorem \ref{sugita}. The classical Wiener space is defined by : 
$$W_0^d=\left\lbrace \omega :[0,1]\to \R^d,\, \text{continuous},\, \omega(0)=0\right\rbrace.  $$ 
This space is equipped with the Borel sigma-field $\mathcal{B}(W_0^d)$ which is derived from the topology of uniform convergence. It is also endowed with the standard Wiener measure $\mu$ \cite{Sugita}. The triplet $(W_0^d,\mathcal{B}(W_0^d),\mu)$ is a probability space and the canonical process defined by :
\begin{equation}\label{BM}
B_t(\omega)=\omega (t),\; \omega \in W_0^d,\, t\in [0,1] 
\end{equation}  
is a Brownian motion. In the sequel, we may use $(W_0^d,\mathcal{B}(W_0^d),\mu)$ as the probability space instead of $\big( \Omega, \sigma(w),\mathbb{P}\big) $ and the process defined by \eqref{BM} as the Brownian motion instead of the process $w(t),\, 0\leqslant t\leqslant 1$.\\

\begin{thm}\label{suta} \cite{Sugita}
If a generalized Wiener functional $\Phi\in \D^{2,-\infty}$ is positive, i.e. 
$$ (\Phi, F)\geqslant 0\quad \forall \; F\in \D^{2,+\infty} \,;\; F(\omega)\geqslant 0 \; \mu\, a.e.$$
then, there exists a unique finite positive measure $\theta_{\Phi} $ on $(W_0^d,\mathcal{B}(W_0^d))$ such that 
$$\forall \; F\in \D^{2,+\infty}, \quad(\Phi, F)=\int F(\omega )\,\theta_{\Phi} (d\omega) \, .$$
\end{thm} 

\noindent The next statement easily follows from theorem \ref{suta}.\\

\begin{thm} 
Let $u\in \R^d,\,u\neq 0$. The limit 
$$ \rho_2 (u)=\lim_{\varepsilon \to 0} \int_{\Delta_2  }  p^d_{\varepsilon}\big( w(t)-w(s)-u\big)dsdt  $$
is a positive generalized Wiener functional. Consequently, there exists a unique finite positive measure $\theta_{u} $ on $(W_0^d,\mathcal{B}(W_0^d))$ such that 
$$\forall \; F\in \D^{2,+\infty},\quad (\rho_2(u), F)=\int F(\omega )\,\theta_{u} (d\omega),.$$
\end{thm} 
$\;$\\
\noindent Next, we focus on the support of the measure $\theta_u$. For this purpose we will use the following lemma.\\

\begin{lem}\cite{Sugita}\label{lema}
Let $A$ be a closed subset of $W_0^d$ and $\delta >0$. Let $A_{\delta}=\cup _{x\in A} B_{(x, \delta)}$ be the $\delta$-neighborhood of $A$ ( $B_{(x, \delta)}$ denotes the open ball with center $x$ and radius $\delta $). There exists $F\in \D^{2,+\infty}$ such that : 
$$ \mathbf{1}_A\leqslant F\leqslant  \mathbf{1}_{A_{\delta}} .$$
\end{lem} 
$\;$

\begin{thm}
Let $u\in \R^d,\,u\neq 0$. The support of the measure $\theta_u$ is included in the closed set : 
$$ E_u=\left\lbrace \phi \in W_0^d\; ; \; \exists \;0\leqslant s<t \leqslant 1,\, \phi(t)-\phi (s)=u \right\rbrace . $$
\end{thm}

\begin{proof}
Let $(\phi_n)_n$ be a sequence of elements from $E_u$ that converges in $ W_0^d$ (uniformly) to $\phi$. For each $n$ there exists $0\leqslant s_n<t_n \leqslant 1$ such that $\phi_n(t_n)-\phi_n (s_n)=u$. Since the sequence $\big((s_n,t_n)\big)_n$ is bounded, there exists a convergent subsequence : $ \big(s_{\tau (n)},t_{\tau(n)}\big) \to (s,t).$ It follows that 
$$ \phi(t)-\phi (s)=\lim_{n\infty}  \phi_{\tau (n)}(t_{\tau (n)})-\phi_{\tau (n)} (s_{\tau (n)})=u\,. $$
(necessarily $s<t$ because $u\neq 0$) So $E_u$ is closed.\\

\noindent Now let $\phi \in W_0^d\setminus E_u$. By continuity of $\phi$, the set $K_{\phi}=\left\lbrace \phi (t)-\phi (s),\, 0\leqslant s\leqslant t \leqslant 1\right\rbrace $ is compact. And since $\phi \notin E_u$ then $u\notin K_{\phi}$. Thus, there exists $r >0$ such that :
$$ \| \phi (t)-\phi (s)-u \| \geqslant r \quad \forall \quad 0\leqslant s\leqslant t \leqslant 1. $$
Let $0<r_1<r_2<r/2$. Then, for each $\omega \in B_{(\phi, r_2) }$ (the open ball of $W_0^d$ with center $\phi$ and radius $r_2$),  
$$ r\leqslant \| \phi (t)-\phi (s)-u \| \leqslant \| \omega (t)-\omega (s)-u \|   + 2\| \omega -\phi \|_{\infty} \leqslant \| \omega (t)-\omega (s)-u \|   + 2r_2 $$
so that 
$$ \| \omega (t)-\omega (s)-u \| \geqslant r -2r_2 \quad \forall \quad  0\leqslant s\leqslant t \leqslant 1. $$

\noindent By virtue of Lemma \eqref{lema}, there exists $F\in \D^{2,+\infty} $ such that : 
$$ \mathbf{1}_{ B_{(\phi, r_1)} }\leqslant F\leqslant  \mathbf{1}_{B_{(\phi, r_2)}} .$$
Consequently,
\begin{align*}
\theta_u \Big( B_{(\phi, r_1)}\Big)
&=\int \mathbf{1}_{ B_{(\phi, r_1) }}(\omega )\,\theta_{u} (d\omega)\\
&\leqslant \int F(\omega )\,\theta_{u} (d\omega)=(\rho_2(u), F) \\
&\leqslant \lim_{\ve \to 0}  \E F(w)\int_{\Delta_2  }  p^d_{\varepsilon}\big( w(t)-w(s)-u\big)dsdt \\
&\leqslant \lim_{\ve \to 0}  \E \mathbf{1}_{ B_{(\phi, r_2) }}(w)\int_{\Delta_2  }  p^d_{\varepsilon}\big( w(t)-w(s)-u\big)dsdt \\
&\leqslant \lim_{\ve \to 0}  \E \mathbf{1}_{ B_{(\phi, r_2) }}(w)\int_{\Delta_2  }  p^d_{\varepsilon}\big( w(t)-w(s)-u\big)dsdt \\
&\leqslant \lim_{\ve \to 0}  \int_{\Delta_2  }  \dfrac{1}{ (2\pi \ve )^{d/2}}\exp \big( -( r -2r_2)^2/(2\ve) \big) dsdt=0\,.
\end{align*}
We deduce : 
$$  \theta_u \Big( B_{(\phi, r_1)}\Big)=0\quad \forall \quad 0<r_1< r/2\,   $$
which finishes the proof.
\end{proof}

$\;$\\
\noindent The next part of this section is devoted to establishing the formula \eqref{integ.eq} on theorem \ref{prop9}. Propositions \ref{propp1}, \ref{propp2} and theorem \ref{thmm} are auxiliary results that will be used in the proof of theorem \ref{prop9}. These auxiliary results are of independent interest and aim to prove that conditional expectation saves the differentiability of random variables as will be stated in theorem \ref{thmm}. We start by recalling the notion of second quantization operator, which will be used in the sequel.\\

\noindent Consider the Hilbert space $ H=L^2\Big( [0,1]; \R^d\Big) $ and let $A$ be a bounded linear operator in $H$ such that $\|A\|\leqslant 1$.\\
\noindent Here we use $(W_0^d,\mathcal{B}(W_0^d),\mu)$ as the probability space. Let $\eta  = \sum _{k=0}^{\infty} I_k(a_k)\, \in L^2\big(W_0^d,\mathcal{B}(W_0^d),\mu \big)  $ .\\
\noindent Each kernel $a_k$ can be identified with a unique symmetric multilinear Hilbert Shmidt form $A_k$ defined on $H^k$. \\
\noindent Let $A_k^A$ be the new symmetric multilinear Hilbert Shmidt form defined by : 
$$ A_k^A (f_1,\cdots, f_k) =A_k (Af_1,\cdots, Af_k) . $$ 
\noindent If we denote by $a_k^A$ the kernel associated with $ A_k^A$ then the second quantization operator $\Gamma (A)$ is defined by : 
$$\Gamma(A)\eta = \sum _{k=0}^{\infty} I_k\big(a_k^A\big). $$
Using the inequality $\big\| I_k\big( a_{k}^A\big) \big \|_2^2\leqslant \big\| A \big \|^{2k} \big\| I_k\big( a_{k}\big) \big \|_2^2$ (here $\big\|. \big \|_2$ denotes the $L^2$-norm), the following proposition becomes straightforward.\\
 
\begin{prop}\label{propp1}\cite{Simon} Let $A$ be a bounded linear operator in $H$ such that $\|A\|\leqslant 1$ and $\eta \in L^2\big( \Omega, \sigma (w),\mathbb{P} \big)$. 
\begin{enumerate}
\item If $\|A\|= 1,\, n\in \N^*$ and $\eta\in \D^{2,n}$ then $\Gamma(A)\eta\in \D^{2,n}$.
\item If $\|A\|< 1$ then $\Gamma(A)\eta\in \D^{2,+\infty}$.
\end{enumerate}
\end{prop}

$\;$\\
\noindent Take a closed subspace $F$ of $ H=L^2\Big( [0,1]; \R^d\Big) $ and associate with it the following sigma-field
$$ \mathcal{A}_F=\sigma \left\lbrace \int_0^1 h_j(t)dw_j(t),\,1\leqslant j\leqslant d,\; h=(h_1,\cdots,h_d)\in F  \right\rbrace. $$
\noindent Let $P$ be the operator of orthogonal projection onto $F$.\\

\begin{prop}\label{propp2}\cite{AADsmoothing} If $\eta \in L^2\big( \Omega, \sigma(w),\mathbb{P}\big)  $, then 
$$
\mathbb{E}\left[\eta \big| \mathcal{A}_F  \right] = \Gamma(P)\eta .
$$
\end{prop}

\noindent The next theorem follows directly from propositions \ref{propp1} and \ref{propp2}.\\
\newpage
\begin{thm} \label{thmm}$\,$ 
\begin{enumerate}
\item If $\eta \in \D^{2,n},\, n\in \N^*$ then $\mathbb{E}\left[\eta \big| \mathcal{A}_F  \right]\in \D^{2,n} $.
\item If $\eta \in \D^{2,+\infty}$ then $\mathbb{E}\left[\eta \big| \mathcal{A}_F  \right]\in \D^{2,+\infty} $.
\end{enumerate}
\end{thm}
$\;$

\begin{thm} \label{prop9}
Let $\eta \in \D^{2,+\infty} $ be bounded. Then, for every $u\in \R^d\setminus \{0\}$,
\begin{equation} \label{integ.eq}
\int_{W_0^d  } \eta(\omega) \theta_{u} (d\omega) = \int_{\Delta_2  } \vp _{s,t}(u)p_{t-s}^d(u)dsdt 
\end{equation}
where $\vp _{s,t}$ is defined by 
$$ \vp _{s,t}\big(w(t)-w(s)\big)=\E \big(\eta\big | w(t)-w(s)\big) $$
and $p_{t-s}^d(u)$ is given by
$$ p_{t-s}^d(u)=\frac{1}{ (2\pi (t-s))^{d/2} } \exp \left( \frac{-\|u\|^2}{2(t-s)} \right). $$
Moreover, both sides in \eqref{integ.eq} are continuous with respect to the variable $u$.
\end{thm}

\begin{proof}
From the definition of the distribution $\rho_2(u)$, the pairing with the test function $\eta$ is given by :
\begin{align*}
\int_{W_0^d} \eta(\omega )\,\theta_{u} (d\omega)
&=(\rho_2(u), \eta)\\
&=\lim_{\varepsilon \to 0} \E \eta\int_{\Delta_2  }  p^d_{\varepsilon}\big( w(t)-w(s)-u\big)dsdt \\
&=\lim_{\varepsilon \to 0} \int_{\Delta_2  }  \E \Big( \eta p^d_{\varepsilon}\big( w(t)-w(s)-u\big)\Big) dsdt \\
&=\lim_{\varepsilon \to 0} \int_{\Delta_2  }  \E \, \E \Big [ \eta p^d_{\varepsilon}\big( w(t)-w(s)-u\big)\Big | w(t)-w(s)\Big ]  dsdt \\
&=\lim_{\varepsilon \to 0} \int_{\Delta_2  }  \E \Big [ p^d_{\varepsilon}\big( w(t)-w(s)-u\big) \vp _{s,t}\big(w(t)-w(s)\big)\Big ]  dsdt \\
&=\lim_{\varepsilon \to 0} \int_{\Delta_2  }  \int_{\R^d  } p^d_{\varepsilon}(x-u) \vp _{s,t}(x) p^d_{t-s}(x)dx\, dsdt \\
&=\lim_{\varepsilon \to 0} \int_{\R^d  } p^d_{\varepsilon}(x-u)  \int_{\Delta_2  }  \vp _{s,t}(x)p^d_{t-s}(x) dsdt \,dx\,.
\end{align*} 
Consider the function 
$$\psi(x)=\int_{\Delta_2  }  \vp _{s,t}(x)p^d_{t-s}(x) dsdt,\; x\in \R^d \setminus \{0\}.$$
Since $\eta \in \D^{+\infty} $ then, by theorem \ref{thmm}, 
$$\vp _{s,t}\big(w(t)-w(s)\big)=\mathbb{E}\left[\eta \big| w(t)-w(s)  \right]\in \D^{2,+\infty}\,. $$
Consequently, 
$$\vp _{s,t}\in \D^{2,+\infty}\Big( \R^d,p_{t-s}^d \Big)\,. $$
Using one of the Sobolev embedding theorems \cite{Adams} we deduce that $\vp _{s,t}(.)$ is continuous. And since $\eta$ is bounded then $\vp _{s,t}(.)$ is bounded independently from $(s,t)$. Therefore, for any $r>0$,
$$ \forall \; x\in \R^d,\; \|x\|\geqslant r \Rightarrow \;  \big | \vp _{s,t}(x)p^d_{t-s}(x)\big |  \leqslant \const \times \dfrac{\exp \Big( -\dfrac{r^2}{2(t-s)} \Big)  }{\big( 2\pi (t-s)\big)^{d/2}} \; \in \, L^1 \big( \Delta_2 \big).$$ 
Now we can use the dominated convergence theorem and conclude that $\psi$ is continuous in $\R^d \setminus \{0\}$. Consequently,
$$\lim_{\varepsilon \to 0} \int_{\R^d  } p^d_{\varepsilon}(x-u)  \psi (x)dx=\psi (u)\,,$$
which finishes the proof.
\end{proof}
$\,$\\
\noindent Using theorem \ref{prop9} and proposition \ref{det.est.} we deduce :\\

\begin{corollary} \label{Cor}
When $u\to 0$, we have
\begin{equation}\label{eq0}
\theta_{u} \Big( W_0^d  \Big)= \int_{\Delta_2  } p_{t-s}^d(u)dsdt  =: m(u,d)\sim   \dfrac{c(d)  }{  \|u\|^{d-2}   }
\end{equation}
where $c(d)$ is a constant that depends only on $d$.
\end{corollary}

$\;$\\
\noindent The last part of this section is devoted to estimating the capacity of the support of the measure $\theta_u$, denoted by $\Supp (\theta_u)$. We begin by recalling the definition of the capacity. The estimate \eqref{finalest} is based on theorem \ref{thma}.\\

\begin{defn}\cite{Sugita}
Let $r >0$. \\
If $O\subset W_0^d$ is open then its $(2,r)$-capacity is defined by :
$$ C_2^{r} (O)=\inf \left\lbrace  \|f\|_{2,r}^2\, ; \, f \in \D^{2,r},\, f\geqslant \mathbf{1}_{O} \right\rbrace . $$
For arbitrary $A\subset W_0^d$, the $(2,r)$-capacity is defined by :
$$ C_2^{r} (A)=\inf \left\lbrace  C_2^{r} (O)\, ; \, A\subset O,\, O \textrm{ is open} \right\rbrace . $$
We say that a subset $A$ of $W_0^d$ is a slim set if $C_p^{r} (A)=0 \quad \forall \; p>1,\; \forall \; r>0$. Here, $C_p^{r}$ is defined as above but with $p$ instead of $2$.
\end{defn}
$\;$
\begin{thm}\cite{Sugita}\label{thma}
Let $r>0$. Let $\Phi \in \D^{2,-r} $ be positive and $\nu_{\Phi}$ its corresponding measure. Then,
$$ \nu_{\Phi}(A)\leqslant \| \Phi \|_{2,-r} \,  \sqrt{ C_2^{r} (A)  } $$
for any $A\in \mathcal{B}(W_0^d)$. In particular, a measure corresponding to a positive generalized Wiener function never has its mass in slim sets.\\  
\end{thm}
$\;$
\begin{thm}
Let $u\in \R^d,\,u\neq 0$ and $\gamma < \frac{4-d}{2}$ . Then the $(2,-\gamma)$-capacity of the support of the measure $\theta_u$ satisfies the inequality 
\begin{equation} \label{finalest}
C_2^{-\gamma}\Big ( \Supp \theta_u \Big) \geqslant c \|u\|^4,\, u\to 0
\end{equation} 
for some constant $c=c(d,\gamma) >0$.
\end{thm}

\begin{proof} Applying theorem \ref{thma} with $\Phi =\rho_2(u), \nu_{\Phi}=\theta_u, r=-\gamma$ and $A=\Supp \theta _u$ we obtain 
$$  \theta_{u}\Big ( \Supp \theta_u \Big) ^2 \leqslant \Big \| \rho _2(u) \Big \|_{2,\gamma} ^2\,  C_2^{-\gamma} \Big ( \Supp \theta_u \Big).$$
From \cite{IPA} we have 
\begin{align*}
\Big \| \rho _2(u) \Big \|_{2,\gamma}^2
&=\sum_{k\geqslant 0} (k+1)^{\gamma} \int _{\Delta_2\times \Delta_2} \dfrac{\lmd \Big( [s_1,t_1]\cap[s_2,t_2]\Big)^k }{ \Big( (t_1-s_1)(t_2-s_2)\Big)^{k/2} } \times\\
& \quad \times \sum _{n_1+\cdots+n_d=k} \prod _{1\leqslant i\leqslant d} \prod _{1\leqslant j\leqslant 2} \dfrac{1}{ \sqrt{n_i!} }\,H_{n_i}\Big( \dfrac{u_i}{\sqrt{ t_j-s_j} } \Big)\, \prod _{1\leqslant j\leqslant 2} p^d_{t_j-s_j}(u) ds_1dt_1ds_2dt_2 
\end{align*} 
Let $\alpha =\alpha (d,\gamma) \in \big ]1/4, 1/2 \big [$ such that $\gamma +d\Big( 1-\dfrac{8\alpha -1}{6} \Big) -3<-1$.\\
From \cite{IPA} we get the following estimations : 
\begin{align*}
& \exists \, c_1=c_1(d,\gamma)\, ; \; \forall \; n\in \N,\, \forall \,x\in \R\; ; \dfrac{|H_n(x)|}{ \sqrt{n!} }\leqslant c_1 (n\vee 1)^{ -\frac{8\alpha -1}{12}} e^{\alpha x^2}, \quad ( n\vee 1=\max \{n,1\}), \\
& \exists \, c_2=c_2(d,\gamma)\, ; \; \forall \; k\in \N\, ; \sum _{n_1+\cdots+n_d=k} \prod _{1\leqslant i\leqslant d} (n_i\vee 1)^{ -\frac{8\alpha -1}{6}} \leqslant c_2 (k\vee 1)^{ d\Big( 1-\frac{8\alpha -1}{6} \Big) -1},  \\
& \exists \, c_3=c_3(d,\gamma)\, ; \; \forall \; x>0\, ;  \dfrac{1}{(2\pi x)^{d/2}} \, \exp \Big( -\big( \frac{1}{2}-\alpha \big) \frac{\|u\|^2}{x}  \Big) \leqslant \dfrac{c_3}{\|u\|^d} \;,   \\
& \exists \, c_4\, ; \; \forall \; k\in \N\, ; \int _{\Delta_2\times \Delta_2} \dfrac{\lmd \Big( [s_1,t_1]\cap[s_2,t_2]\Big)^k }{ \Big( (t_1-s_1)(t_2-s_2)\Big)^{k/2} } ds_1dt_1ds_2dt_2 \leqslant c_4 (k\vee 1) ^{-2}    \,.
\end{align*}
Using these inequalities we deduce the existence of constants $c_5=c_5(d,\gamma),\, c_6=c_6(d,\gamma) $ such that 
$$\Big \| \rho _2(u) \Big \|_{2,\gamma}^2 \leqslant \dfrac{c_5}{ \|u\|^{2d} }\sum_{k\geqslant 0} (k+1)^{\gamma}  (k\vee 1)^{ d\big( 1-\frac{8\alpha -1}{6} \big) -3}\leqslant \dfrac{c_6}{ \|u\|^{2d} }\,.$$
Consequently, 
$$  \theta_{u}\Big ( \Supp \theta_u \Big) ^2 \leqslant \dfrac{c_6}{ \|u\|^{2d} }\,  C_2^{-\gamma} \Big ( \Supp \theta_u \Big).$$
From another side, using corollary \ref{Cor}, we obtain  
$$ \theta_{u}\Big ( \Supp \theta_u \Big) \sim   \dfrac{c_7  }{  \|u\|^{d-2}   }\,,\; u\to 0, $$
for some constant $c_7=c_7(d)$, which finishes the proof.
\end{proof}

\section{Estimation of the quadratic Wasserstein distance}
In this section we will try to describe the behavior of the family of measures $\theta_u,\, u \in \R^d, u\neq 0, d\geqslant 4$, obtained in section 2 in relation with the $d$-dimensional Brownian motion $W(t)$, in terms of the quadratic Wasserstein distance between $\theta_u$ and the Wiener measure $\mu$. Let us first recall the definition of quadratic Wasserstein distance.\\

\begin{defn}\cite{CV}
Let $(M,d)$ be a metric space equipped with its Borel sigma-field. Let $\nu_1,\nu_2$ be two probability measures in $M$ which have finite second moment, i.e. 
$$ \int_M d^2(x,y)\nu_i(dy) <\infty,\; i=1,2,\; \text{ for some (and so for all ) } x\in M.  $$
The quadratic Wasserstein distance between $\nu_1$ and $\nu_2$ is then defined by : 
$$\mathcal{W}_2^2 (\nu_1,\nu_2)=\inf_{\nu}  \int_M \int_M d^2(x,y)\nu (dx,dy) $$
where the infimum is taken over all couplings $\nu$ of $\nu_1$ and $\nu_2$, i.e. all probability measures $\nu$ on $M\times M$ with marginals $\nu_1$ and $\nu _2$.
\end{defn}

\noindent Before we start estimating the quadratic Wasserstein distance between $\theta_u$ and the Wiener measure $\mu$, let us check that they have finite second moment. For the Wiener measure, we will use the Fernique theorem.

\begin{thm}\cite{Fernique}
Let $\nu$ be a Gaussian probability measure on some Banach space $(X, \| .\| )$. Then, there exists $\alpha >0$ such that 
$$ \int _X e^{\alpha \| x\| ^2} \nu (dx) <\infty.  $$
\end{thm}
\noindent Consequently, any Gaussian probability measure, and in particular the Wiener measure, on a Banach space has finite moments of any order.\\
\noindent Now, to check that also the measures $\theta_u$ have finite second moment, we will use a Fernique-type theorem.
\begin{thm}\cite{Sugita}
Let $\Phi \in \D^{2,-r}(W_0^d) $ be positive and $\nu_{\Phi}$ its corresponding measure. Then, there exists $\alpha >0$ such that 
$$ \int _{W_0^d} e^{\alpha \| x \| ^2} \nu_{\Phi} (dx) <\infty.  $$
\end{thm}
\noindent It follows from this theorem that the measures $\theta_u$ have finite moments of any order. \\
\noindent Notice that $\theta_u$ is not a probability measure. Let us then define a probability measure
$$ \widetilde{\theta_u}=\dfrac{\theta_u}{ \theta_u(W_0^d) }=\dfrac{\theta_u }{m(u,d) } \,.$$
\noindent In order to compute the quadratic Wasserstein distance between $\theta_u$ and the Wiener measure $\mu$ we will use the "Talagrand inequality". But first let us recall a definition.
\begin{defn}\cite{CV}
Let $p,q$ be two probability measures on the same probability space $X$. The relative entropy of $p$ with respect to $q$ is defined by : 
$$ H(p|q)=\int _{X} \log \Big( \dfrac{dp}{dq} \Big) dp $$
whenever $p$ is absolutely continuous with respect to $q$, otherwise $H(p|q)=\infty $.
\end{defn}

\begin{thm}\cite{Blower}
Let $\nu_1,\nu _2$ be two probability measures on $\R^N$. Suppose that $\nu_2$ has a density with respect to Lebesgue measure of the form : $\nu_2(dx)=e^{-V(x)}dx$ where $V :\R^N\to [0,\infty) $ and $V\in C^{\infty} ( \R^N)$. Suppose that there exists $\kappa >0$ such that $Hess( V)\geqslant \kappa . I_N$.  Then, 
$$ \mathcal{W}_2^2 (\nu_1,\nu_2)\leqslant \dfrac{2}{\kappa} H(\nu_1|\nu_2). $$
\end{thm}

\noindent In order to use this inequality, we need to consider finite dimensional marginals of $\widetilde{\theta_u}$ and $\mu$. 
\begin{defn}
Let $0=t_0<t_1\cdots <t_n\leqslant 1$ and $\vec{t}=(t_1,\cdots,t_n)$. There exists a unique positive generalized Wiener function  $\rho_{2,\vec{t}}(u)$ on $\R^{d\times n}$ such that for every $f\in C_b^{\infty} (\R^{d\times n})$ we have 
$$ \Big( f\big( W(t_1),\cdots, W(t_n) \big), \rho_{2,\vec{t}}(u)\Big) =\Big( f\big( W(t_1),\cdots, W(t_n) \big), \rho_{2}(u)\Big) .  $$
We denote by $ \theta_{u,\vec{t}} $ the associated measure in $\R ^{d\times n}$.\\
If $ \vec{t}=(\frac{1}{n},\frac{2}{n},\cdots,\frac{n-1}{n},1) $ then $ \theta_{u,\vec{t}} $ is denoted $ \theta_{u,n} $.
\end{defn}
\begin{defn}
Let $0=t_0<t_1\cdots <t_n\leqslant 1$ and $\vec{t}=(t_1,\cdots,t_n)$. We denote by $ \mu_{\vec{t}}$ the probability measure in $\R ^{d\times n}$ of $\big( W(t_1),\cdots, W(t_n) \big)$. We recall that 
$$ \mu_{\vec{t}}(dx_1,\cdots,dx_n)=\prod _{j=1}^n p^d_{t_j-t_{j-1}}(x_j-x_{j-1})\, dx_1\cdots dx_n\,. $$
If $ \vec{t}=(\frac{1}{n},\frac{2}{n},\cdots,\frac{n-1}{n},1) $ then $ \mu_{\vec{t}} $ is denoted $ \mu_{n} $.
\end{defn}

\begin{lem} We have 
$$ \theta_{u,\vec{t}} \big( \R ^{d\times n} \big)=m(u,d) . $$
So that $ \widetilde{\theta_{u,\vec{t}}}=\dfrac{\theta_{u,\vec{t}}}{  m(u,d)  }  $ is a probability measure on $\R ^{d\times n}$.
\end{lem}
\begin{proof}
By definition, 
$$ \theta_{u,\vec{t}} \big( \R ^{d\times n} \big)=\Big( 1, \rho_{2,\vec{t}}(u)\Big) =\Big( 1, \rho_{2}(u)\Big)=\theta_u(W_0^d)  =m(u,d) . $$
\end{proof}

\begin{lem} For every $s<t$ we have 
$$ \E \Big [ p_{\ve} ^d \big(W(t)-W(s)-u \big)\Big| W(t_1),\cdots, W(t_n) \Big]  =p_{\ve + \sigma^2} ^d \Big( \sum_{j=1}^n \dfrac{\alpha_j}{t_j-t_{j-1}} (W(t_j)-W(t_{j-1}))\;-u \Big), $$
where 
\begin{align*}
& \alpha_j= \lambda \big( [s,t]\cap [t_{j-1},t_j]\big) \,, \; j=1,\cdots, n\\
& \sigma^2=t-s-\sum_{j=1}^n \dfrac{\alpha_j ^2}{t_j-t_{j-1}}= dist ^2 \Big( \mathbf{1}_{[s,t]},\, LS\Big(  \mathbf{1}_{[t_0,t_1]},\cdots, \mathbf{1}_{[t_{n-1},t_n]} \Big) \Big).
\end{align*}
\end{lem}

\begin{proof}
Decompose the vector $W(t)-W(s)$ as follows  
$$\begin{cases}
W(t)-W(s)=X+Y \\
X=  \text{ orthogonal projection of } W(t)-W(s) \text{onto } LS(W(t_1),\cdots, W(t_n)) : \\
X=\sum_{j=1}^n  \dfrac{\alpha_j}{t_j-t_{j-1}} (W(t_j)-W(t_{j-1}))\,, \\
Y\in LS(W(t_1),\cdots, W(t_n)) ^{\perp}.
\end{cases}$$
It follows that $Y$ is independent from $(W(t_1),\cdots, W(t_n))$. Consequently, 
\begin{align*}
\E \Big [ p_{\ve} ^d \big(W(t)-W(s)-u \big)\Big| W(t_1),\cdots, W(t_n) \Big] 
&=\E \Big [ p_{\ve} ^d \big(X+Y-u \big)\Big| W(t_1),\cdots, W(t_n) \Big]   \\
&=\E  p_{\ve} ^d \big(x+Y-u \big)\Big | _{x=X}\\
&= p_{\ve +\sigma ^2} ^d \big(x -u \big) \Big | _{x=X} \\
&= p_{\ve +\sigma ^2} ^d \big(X -u \big)  \\
&=p_{\ve + \sigma^2} ^d \Big( \sum_{j=1}^n \dfrac{\alpha_j}{t_j-t_{j-1}} (W(t_j)-W(t_{j-1}))\;-u \Big).
\end{align*}
\end{proof}

\begin{lem} The measure $ \theta_{u,\vec{t}} $ has a density with respect to $ \mu_{\vec{t}}$ given by 
$$ q_{u,\vec{t}}(x_1,\cdots, x_n)=\int_{\Delta_2  } p_{\sigma ^2}^d \Big( \sum_{j=1}^n \dfrac{\alpha_j}{t_j-t_{j-1}} (x_j-x_{j-1})\;-u \Big)  dsdt \,.$$
If $ \vec{t}=(\frac{1}{n},\frac{2}{n},\cdots,\frac{n-1}{n},1) $ then $ q_{u,\vec{t}}$ is denoted $ q_{u,n} $ :
$$ q_{u,n}(x_1,\cdots, x_n)=\int_{\Delta_2  } p_{\sigma ^2}^d \Big( n\sum_{j=1}^n \alpha_j (x_j-x_{j-1})\;-u \Big)  dsdt \,.$$
\end{lem}
\begin{proof}
From the definition of $ \theta_{u,\vec{t}} $, for any $f\in C_b^{\infty} (\R^{d\times n})$ we have

\begin{align*}
&\int_{ \R^{d\times n} } f(x_1,\cdots,x_n) \theta_{u,\vec{t}}(dx_1,\cdots,dx_n)\\
&=\Big( f\big( W(t_1),\cdots, W(t_n) \big), \rho_{2,\vec{t}}(u)\Big) \\
&=\Big( f\big( W(t_1),\cdots, W(t_n) \big), \rho_{2}(u)\Big)\\
&=\lim_{\varepsilon \to 0} \E f\big( W(t_1),\cdots, W(t_n) \big) \int_{\Delta_2  }  p^d_{\varepsilon}\big( W(t)-W(s)-u\big)dsdt \\
&=\lim_{\varepsilon \to 0} \int_{\Delta_2  }  \E \, \E \Big [  f\big( W(t_1),\cdots, W(t_n) \big) p^d_{\varepsilon}\big( W(t)-W(s)-u\big)\Big| W(t_1),\cdots, W(t_n) \Big]  dsdt\\
&=\lim_{\varepsilon \to 0} \int_{\Delta_2  }  \E \Bigg( f\big( W(t_1),\cdots, W(t_n) \big)  \E \Big [  p^d_{\varepsilon}\big( W(t)-W(s)-u\big)\Big| W(t_1),\cdots, W(t_n) \Big] \Bigg) dsdt\\
&=\lim_{\varepsilon \to 0} \int_{\Delta_2  }  \E \Bigg( f\big( W(t_1),\cdots, W(t_n) \big)  p_{\ve + \sigma^2} ^d \Big( \sum_{j=1}^n \dfrac{\alpha_j}{t_j-t_{j-1}} (W(t_j)-W(t_{j-1}))\;-u \Big) \Bigg) dsdt\\
&=\int_{\Delta_2  }  \E \Bigg( f\big( W(t_1),\cdots, W(t_n) \big)  p_{ \sigma^2} ^d \Big( \sum_{j=1}^n \dfrac{\alpha_j}{t_j-t_{j-1}} (W(t_j)-W(t_{j-1}))\;-u \Big) \Bigg) dsdt\\
&=\int_{ \R^{d\times n} } f(x_1,\cdots,x_n) \int_{\Delta_2  }  p_{ \sigma^2} ^d \Big( \sum_{j=1}^n \dfrac{\alpha_j}{t_j-t_{j-1}} (x_j -x_{j-1})\;-u \Big)  dsdt \; \mu_{\vec{t}} (dx_1,\cdots dx_n)\,.
\end{align*}
Lemma is proved.
\end{proof}

\begin{lem} The measure $ \mu_{n}$ has a density of the form $ e^{-V(x)}dx$ where $V :\R^{d\times n}\to [0,\infty) $ and $V\in C^{\infty} ( \R^{d\times n})$. The Hessian matrix of $V$ has eigenvalues 
$$ \lambda_j=2n\Big( 1-\cos \frac{2j+1}{2n+1}\pi\,\Big),\; j=0,\cdots,n-1. $$
In particular, the best constant $\kappa_n >0$ such that $Hess( V)\geqslant \kappa_n . I_{n\times d}$ is given by 
$$ \kappa_n=\min_j \lambda_j =\lambda_0=2n\Big( 1-\cos \dfrac{\pi}{2n+1}\,\Big). $$
\end{lem}

\begin{proof}
Recall that the density of $ \mu_{n}$ is $\prod _{j=1}^n p^d_{ \frac{1}{n} }(x_j-x_{j-1})$, so that the Hessian matrix of $V$ is the tridiagonal matrix  
$$ Hess( V)=\begin{bmatrix}
2nI_d      & -nI_d    & 0        & \cdots &   \cdots     & 0  \\ 
-nI_d      & 2nI_d    & -nI_d    & \ddots &        &  \\ 
0          & -nI_d    & 2nI_d    & \ddots &    \ddots     & \vdots \\
 \vdots & \ddots   & \ddots & \ddots & \ddots &  0 \\
 \vdots &   & \ddots  & -nI_d      & 2nI_d      & -nI_d \\
0    &\cdots        & \cdots & 0      & -nI_d      & nI_d 
\end{bmatrix} \in \mathcal{M}\big( n\times d,\R \big). $$
Using matrix tensor product we can write : $ Hess( V)=A\otimes I_d $ where 
$$ A=\begin{bmatrix}
2n       & -n     & 0        & \cdots &   \cdots     & 0  \\ 
-n       & 2n     & -n     & \ddots &        &  \\ 
0          & -n   & 2n    & \ddots &    \ddots     & \vdots \\
 \vdots & \ddots   & \ddots & \ddots & \ddots &  0 \\
 \vdots &   & \ddots  & -n      & 2n       & -n \\
0    &\cdots        & \cdots & 0      & -n       & n  
\end{bmatrix} \in \mathcal{M}\big( n ,\R \big). $$
It follows that the eigenvalues of $Hess(V)$ are the products of eigenvalues of $A$ and those of $I_d$, which means that $Hess(V)$ and $A$ have the same eigenvalues.\\
Now, $A=2nI_n-nB$, where 
$$B=\begin{bmatrix}
0       & 1     & 0        & \cdots &   \cdots     & 0  \\ 
1       & 0     & 1     & \ddots &        &  \\ 
0          & 1   & 0    & \ddots &    \ddots     & \vdots \\
 \vdots & \ddots   & \ddots & \ddots & \ddots &  0 \\
 \vdots &   & \ddots  & 1      & 0       & 1 \\
0    &\cdots        & \cdots & 0      & 1       & 1  
\end{bmatrix} \in \mathcal{M}\big( n ,\R \big). $$
By solving a system of linear equations $BZ=\lmd Z$ we find the eigenvalues of $B$ 
$$ 2\cos \frac{2j+1}{2n+1}\pi\, ,\; j=0,\cdots,n-1.  $$
Finally, the Hessian matrix of $V$ has eigenvalues 
$$ \lambda_j=2n\Big( 1-\cos \frac{2j+1}{2n+1}\pi\,\Big),\; j=0,\cdots,n-1, $$
which finishes the proof.
\end{proof}

\begin{lem}\label{lem60} We have 
$$ \int _{\R^{d\times n}} q_{u,n}(x_1,\cdots, x_n)\,\mu_{n}(dx_1,\cdots,dx_n)=m(u,d).  $$
 
\end{lem}

\begin{proof}
Let 
\begin{align*}
J_n
&=\int _{\R^{d\times n}} q_{u,n}(x_1,\cdots, x_n)\,\mu_{n}(dx_1,\cdots,dx_n) \\
&=\int _{\R^{d\times n}} \int_{\Delta_2  } p_{ \sigma^2} ^d \Big( n\sum_{j=1}^n  \alpha_j (x_j -x_{j-1})\;-u \Big)ds dt \; \prod _{j=1}^n p^d_{ \frac{1}{n} }(x_j-x_{j-1}) dx_1\cdots dx_n\\
&=\int_{\Delta_2  }  \int _{\R^{d\times n}} p_{ \sigma^2} ^d \Big( n\sum_{j=1}^n  \alpha_j (x_j -x_{j-1})\;-u \Big) \prod _{j=1}^n p^d_{ \frac{1}{n} }(x_j-x_{j-1}) dx_1\cdots dx_n \; ds dt\,.
\end{align*} 
By making a change of variables : $u_j=x_j-x_{j-1},\, j=1,\cdots,n $ we obtain 
\begin{align*}
J_n
&=\int_{\Delta_2  }  \int _{\R^{d\times n}} p_{ \sigma^2} ^d \Big( n\sum_{j=1}^n  \alpha_j u_j\;-u \Big) \prod _{j=1}^n p^d_{ \frac{1}{n} }(u_j ) du_1\cdots du_n \; ds dt\\
&=\int_{\Delta_2  }  \int _{\R^{d\times (n-1)}} \prod _{j=1}^{n-1} p^d_{ \frac{1}{n} }(u_j )  \int _{\R^{d }} p_{ \sigma^2} ^d \Big( n\sum_{j=1}^{n }  \alpha_j u_j\;-u \Big)  p^d_{ \frac{1}{n} }(u_n ) du_n \; du_1\cdots du_{n-1} \; ds dt\\
&=\dfrac{1}{(n\alpha_n)^d} \int_{\Delta_2  }  \int _{\R^{d\times (n-1)}} \prod _{j=1}^{n-1} p^d_{ \frac{1}{n} }(u_j )  \int _{\R^{d }} p_{ \frac{\sigma^2}{ n^2\alpha_n^2 } }^d \Big( u_n+ \frac{1}{\alpha _n }\sum_{j=1}^{n-1 }  \alpha_j u_j\;-\frac{1}{\alpha _n } u \Big)  p^d_{ \frac{1}{n} }(u_n ) du_n  du_1\cdots du_{n-1} ds dt\\
&= \int_{\Delta_2  }  \int _{\R^{d\times (n-1)}} \prod _{j=1}^{n-1} p^d_{ \frac{1}{n} }(u_j )  p_{  \sigma^2 +n\alpha _n^2 }^d \Big( n\sum_{j=1}^{n-1 }  \alpha_j u_j\;- u \Big)   du_1\cdots du_{n-1} \; ds dt.
\end{align*} 
By induction, we obtain
\begin{align*}
J_n
&= \int_{\Delta_2  }  \int _{\R^{d }}  p^d_{ \frac{1}{n} }(u_1 )  p_{  \sigma^2 +n\sum_{j=2}^{n }\alpha _j^2 }^d \Big( n\alpha_1 u_1\;- u \Big)   du_1  \; ds dt\\
&=\int_{\Delta_2  } p_{t-s}^d(u)ds dt.
\end{align*} 
Proof is finished.
\end{proof}

\begin{lem} The relative entropy of $\widetilde{\theta_{u,n}}$ with respect to $\mu_n$ is such that 
$$ H(\widetilde{\theta_{u,n}}|\mu_n)\leqslant  - \log\Big( 2m(u,d) (2\pi)^{d/2}\Big)   -\dfrac{d}{2m(u,d)} \int_{\Delta_2  } \log \big(\sigma ^2 \big) p_{t-s}^d ( u) dsdt \,. $$
\end{lem}

\begin{proof}
\begin{align*}
H(\widetilde{\theta_{u,n}}|\mu_n)
&= \int _{\R^{d\times n}}  \log \Big( \dfrac{d\widetilde{\theta_{u,n}}}{d\mu_n} \Big) \; d\widetilde{\theta_{u,n}}\\
&= \int _{\R^{d\times n}}  \log \Big( \dfrac{q_{u,n}}{m(u,d)} \Big) \; \dfrac{q_{u,n}}{m(u,d)}d\mu_n \\
&= \int _{\R^{d\times n}}  \log \Bigg( \dfrac{1}{m(u,d)} \int_{\Delta_2  } p_{\sigma ^2}^d \Big( n\sum_{j=1}^n  \alpha_j (x_j-x_{j-1})\;-u \Big)  dsdt  \Bigg) \times \\
& \phantom{= \;} \dfrac{1}{m(u,d)} \int_{\Delta_2  } p_{\sigma ^2}^d \Big( n\sum_{j=1}^n  \alpha_j (x_j-x_{j-1})\;-u \Big)  dsdt  \; \prod _{j=1}^n p^d_{ \frac{1}{n} }(x_j-x_{j-1}) dx_1\cdots dx_n\,.
\end{align*}
Now we use Jensen inequality for the convex function $x\mapsto x\log x$: 
\begin{align*}
H(\widetilde{\theta_{u,n}}|\mu_n)
&= \int _{\R^{d\times n}}  \log \Bigg( \dfrac{1}{2m(u,d)} \int_{\Delta_2  } p_{\sigma ^2}^d \Big( n\sum_{j=1}^n  \alpha_j (x_j-x_{j-1})\;-u \Big)  2dsdt  \Bigg) \times \\
& \phantom{= \;} \dfrac{1}{2m(u,d)} \int_{\Delta_2  } p_{\sigma ^2}^d \Big( n\sum_{j=1}^n  \alpha_j (x_j-x_{j-1})\;-u \Big)  2dsdt  \; \prod _{j=1}^n p^d_{ \frac{1}{n} }(x_j-x_{j-1}) dx_1\cdots dx_n\\
&\leqslant \int _{\R^{d\times n}}  \int_{\Delta_2  } \log \Bigg( \dfrac{1}{2m(u,d)} p_{\sigma ^2}^d \Big( n\sum_{j=1}^n  \alpha_j (x_j-x_{j-1})\;-u \Big)  \Bigg) \times \\
& \phantom{= \;} \dfrac{1}{2m(u,d)} p_{\sigma ^2}^d \Big( n\sum_{j=1}^n  \alpha_j (x_j-x_{j-1})\;-u \Big)  2dsdt  \; \prod _{j=1}^n p^d_{ \frac{1}{n} }(x_j-x_{j-1}) dx_1\cdots dx_n\,.
\end{align*}
Using the estimate  
$$ \log \Bigg( \dfrac{1}{2m(u,d)} p_{\sigma ^2}^d \Big( n\sum_{j=1}^n  \alpha_j (x_j-x_{j-1})\;-u \Big)  \Bigg) \leqslant -\log \Big(  2m(u,d) \big( 2\pi \big)^{d/2} \Big) -\dfrac{d}{2} \log \sigma ^2\,, $$
we deduce, using lemma \eqref{lem60}, 
\begin{align*}
H(\widetilde{\theta_{u,n}}|\mu_n)
&\leqslant \int _{\R^{d\times n}}  \int_{\Delta_2  } \log \Bigg( \dfrac{1}{2m(u,d)} p_{\sigma ^2}^d \Big( n\sum_{j=1}^n  \alpha_j (x_j-x_{j-1})\;-u \Big)  \Bigg) \times \\
& \phantom{= \;} \dfrac{1}{2m(u,d)} p_{\sigma ^2}^d \Big( n\sum_{j=1}^n  \alpha_j (x_j-x_{j-1})\;-u \Big)  2dsdt  \; \prod _{j=1}^n p^d_{ \frac{1}{n} }(x_j-x_{j-1}) dx_1\cdots dx_n\\
&\leqslant -\int _{\R^{d\times n}}  \int_{\Delta_2  } \Bigg( \log \Big(  2m(u,d) \big( 2\pi \big)^{d/2} \Big) +\dfrac{d}{2} \log \sigma ^2 \Bigg) \times \\
& \phantom{= \;} \dfrac{1}{2m(u,d)} p_{\sigma ^2}^d \Big( n\sum_{j=1}^n  \alpha_j (x_j-x_{j-1})\;-u \Big)  2dsdt  \; \prod _{j=1}^n p^d_{ \frac{1}{n} }(x_j-x_{j-1}) dx_1\cdots dx_n\\
&\leqslant -\int_{\Delta_2  } \Bigg( \log \Big(  2m(u,d) \big( 2\pi \big)^{d/2} \Big) +\dfrac{d}{2} \log \sigma ^2 \Bigg) \times \\
& \phantom{= \;} \dfrac{1}{m(u,d)} \int _{\R^{d\times n}}  p_{\sigma ^2}^d \Big( n\sum_{j=1}^n  \alpha_j (x_j-x_{j-1})\;-u \Big)  \prod _{j=1}^n p^d_{ \frac{1}{n} }(x_j-x_{j-1}) dx_1\cdots dx_n \; dsdt  \\
&\leqslant -\int_{\Delta_2  } \Bigg( \log \Big(  2m(u,d) \big( 2\pi \big)^{d/2} \Big) +\dfrac{d}{2} \log \sigma ^2 \Bigg) p_{t-s}^d ( u) ds dt\\
&\leqslant - \log \Big(  2m(u,d) \big( 2\pi \big)^{d/2} \Big) -\dfrac{d}{2 m(u,d)} \int_{\Delta_2  } \log (\sigma ^2) p_{t-s}^d ( u)ds dt.
\end{align*}
Lemma is proved.
\end{proof}

\noindent Combining previous results we obtain the following proposition.

\begin{prop} The quadratic Wasserstein distance between $\widetilde{\theta_{u,n}}$ and $\mu _n$ satisfies the estimate
$$ \mathcal{W}_2^2 \Big( \widetilde{\theta_{u,n}},\mu _n \Big)\leqslant \dfrac{2}{2n\Big( 1-\cos \dfrac{\pi}{2n+1}\,\Big)}  \Bigg( - \log\Big( 2m(u,d) (2\pi)^{d/2}\Big)   -\dfrac{d}{2m(u,d)} \int_{\Delta_2  } \log \big(\sigma ^2 \big) p_{t-s}^d ( u) dsdt \Bigg).  $$
\end{prop}

\end{document}